\documentclass[a4paper,reqno]{article}
\usepackage{amsmath,amsthm,amssymb,amsfonts,euscript,graphicx}
\usepackage{multicol,multirow}
\usepackage{pst-node,pst-plot,tikz}
\usetikzlibrary{shapes,snakes}
\usepackage{mathdots}
\pstVerb{realtime srand}
\psset{plotpoints=50}
\usepackage{qtree}
\usepackage{algorithm}
\usepackage[noend]{algpseudocode}
\makeatletter
\def\BState{\State\hskip-\ALG@thistlm}
\makeatother

\theoremstyle{plain}
\newtheorem{theorem}{Theorem}
\newtheorem{lemma}{Lemma}
\newtheorem{corollary}{Corollary}
\newtheorem{proposition}{Proposition}

\theoremstyle{definition}
\newtheorem{definition}{Definition}

\newtheorem{remark}{Remark}

\newtheorem{conjecture}{Conjecture}

\begin{document}

\title{\large \bf On matchable subsets in abelian groups and their linear analogues
}

\author{{\normalsize{\sc Mohsen Aliabadi}$^{a,1}$ and {\sc Mano Vikash Janardhanan}$^{a,2}$}\\
{\footnotesize{\it $^a$Department of Mathematics, Statistics, and Computer Science, University of Illinois at Chicago,   }}\\
{\footnotesize{\it 851 S. Morgan St, Chicago, IL 60607, USA}}\\
{\footnotesize{$^1$E-mail address: $\mathsf{maliab2@uic.edu}$}}\\
{\footnotesize{$^2$E-mail address: $\mathsf{mjanar2@uic.edu}$}}\\
}
\date{}

\maketitle
\begin{abstract}
In this paper, we introduce the notions of matching matrices in groups and vector spaces, which lead to some necessary conditions for existence of acyclic matching in abelian groups and its linear analogue. We also study the linear local matching property in field extensions to find a dimension criterion for linear locally matchable bases. Moreover, we define the weakly locally matchable subspaces and we investigate their relations with matchable subspaces. We provide an upper bound for the \linebreak dimension of  primitive subspaces in a separable  field extension. We  employ  MATLAB coding to investigate the  existence of acyclic matchings in finite cyclic groups. Finally,   a possible research  problem on matchings in $n$-groups is presented. Our tools in this paper mix combinatorics and linear algebra.
\end{abstract}

\tableofcontents

\section{Introduction}
Let $B$ be a finite subset of the group $\mathbb{Z}^n$ which does not contain the neutral element. For any subset $A$ in $\mathbb{Z}^n$ with the same cardinality as $B$, a {\it matching} from $A$ to $B$ is defined to be a bijection $\varphi:A\to B$ such that for any $a\in A$ we have $a+\varphi(a)\not\in A$. For any matching $\varphi$ as above, the associated multiplicity function $m_\varphi:\mathbb{Z}^n\to \mathbb{Z}_{\geq0}$ is defined via the rule:
\begin{align}
\forall x\in \mathbb{Z}^n,\quad m_\varphi(x)=\#\{a\in A:\, a+\varphi(a)=x\}.
\end{align}
A matching $\varphi:A\to B$ is called {\it acyclic} if for any matching $g:A\to B$, $m_f=m_g$ implies $f=g$. The notion of matchings in groups was introduced by Fan and Losonczy in \cite{*} in order to generalize a geometric property of lattices in Euclidean space.  Fixing finite subsets $A$ and $B$ in $\mathbb{Z}^n$ with the same cardinality such that $0\not\in B$, the existence of an acyclic matching from $A$ to $B$ is proven \cite{4}. The motivation to study acyclic matchings is their relations with an old problem of Wakeford concerning canonical forms for symmetric tensors \cite{12}. In other words, acyclic matchings were used to prove that any small enough fixed set of monomials can be removed from a generic homogeneous form after a suitable linear change of variables. The notion of matchings was generalized and explored in the context of arbitrary groups \cite{6}. Later, the notion of the local matching property was introduced in 2018 Aliabadi-Janardhanan to study the matching property under weaker conditions \cite{3}. The main purpose of this paper is to proceed to study the ``main problem concerning the relation between matchable bases and locally matchable bases".\\
It is shown in \cite{9} that any torsion-free abelian group possesses the acyclic matching property. In Section 2 we introduce the matching matrices to study acyclic matchings in any abelian groups, not torsion-free necessarily. We also define the weak acyclic matching property to characterize finite cyclic groups in terms of acyclic matching property in usual sense.  Then we provide the linear analogue of matching matrices  to study the linear acyclic matching property in vector spaces.  Section 3 presents  a dimension criterion to study the linear local matching property in subspaces in a field extension. We would like to mention that although the main problem concerning the relation between matchable bases and locally matchable bases  is still unsolved (\cite[Remak 5.7]{3}), our dimension criterion should be useful to   characterize it. We introduce the concept of weakly locally matchable subspaces which is very similar to locally matchable subspace in its usual sense. We establish that matchable subspaces are weakly locally matchable as well.  In Section 4, we provide a formula to estimate the dimension of primitive subspaces in a separable field extension motivated by matchable bases problems in primitive vector subspaces \cite[Theorem 4.2]{3}.  In Section 5, we give MATLAB coding of our conjectures in Section 2 concerning the acyclic matching property and the weak acyclic matching property in finite cyclic groups and discuss simulation results.  Finally, in Section 6  a possible research problem  on matchings in $n$-groups is presented.

\section{Acyclicity}
In this section, we introduce  matching matrices in abelian groups and prove that the matching matrix of an acyclic  matching  is invertible. Then we introduce the weak  acyclic matching property in abelian groups and provide an open problem similar to those in \cite{2} regarding characterizing prime numbers in terms of acyclicity. Next, we define the linear analogue of matching matrices in field extensions and we conjecture the linear version of  matching matrices in the field setting. We  provide an open problem concerning the existence of the linear acyclic matching property in field extensions motivated by its analogue to \cite[Proposition 2.2 and Proposition 2.3]{2}. We finally investigate matchable subspaces in field extensions that contain proper finite dimensional intermediate subfields.\\
Following Losonczy in \cite{9}, we say that a group $G$ possesses the {\it matching property} if for every pair $A$ and $B$ of finite subsets of $G$ with   $\# A=\#B$ and $0\not\in B$ there is at least one  matching from $A$ to $B$. Also, $G$ possesses the {\it acyclic matching property} if for every pair $A$ and $B$ of finite subsets of $G$ with   $\# A=\#B$ and $0\not\in B$ there is at least one  acyclic matching from $A$ to $B$. The main result concerning the acyclic matching property is proven by Losonczy in \cite{9}, where it is shown that any torsion-free abelian group has the acyclic matching property.\\
Let $G$ be an abelian group. For each $g\in G$ we associate a variable $x_g$. Consider now the commutative ring of polynomials with integer coefficients in all possible variables. Call it $D(G)$.
Let $A=\{a_1,\ldots,a_n\}$ and $B=\{b_1,\ldots,b_n\}$ be two subsets of $G$. Define the $n\times n$ matrix $M_{A,B}=[m_{ij}]_{n\times n}$ as follows:
\[m_{ij}=\begin{cases}x_{a_i+b_j}& \mathrm{if}\;\; a_i+b_j\not\in A;\\ 0& \mathrm{otherwise}.\end{cases}\]
\begin{definition}\label{d2.1}
The matrix $M_{A,B}$ is called the matching matrix from $A$ to $B$.
\end{definition}
Note that $M_{A,B}$ is defined over the domain $D(G)$. More precisely, let $S=A+B$. Then $D(S)$ is the finitely generated  subring in the variables in $S$.\\
In the following observation we show that the necessary condition for the existence of an acyclic matching from $A$ to $B$ is that $M_{A,B}$ is invertible. Note that here we mimick Losonczy's proof \cite{*} concerning removable subsets of $T(q,p)$, where $T(q,p)$ denotes the set of all $q$-tuples $I=(i_1,\ldots,i_q)$ of nonnegative integers satisfying \[i_1+i_2+\cdots+i_q=p.\]
\begin{proposition}\label{t2.2}
Let $A$ and $B$ be two finite subsets of the abelian group  $G$. If there is an acyclic matching from $A$ to $B$,  then $M_{A,B}$ is invertible.
\end{proposition}
\begin{proof}
The determinant of $M_{A,B}$ equals, up to sign,
\begin{align}
\sum_\sigma\left((-1)^\sigma\cdot\prod_{\substack{1\leq i\leq n\\ b_j=\sigma(a_i)\\\sigma\in\mathcal{S}_n}}m_{i,\sigma(j)}\right),
\end{align}
where $\mathcal{S}_n$ stands for the symmetric group of degree $n$. There is a one to one correspondence between matchings from $A$ to $B$ and the nonzero summands in the above expansion. Assume that $\varphi:A\to B$ is an acyclic matching. The summand corresponding to $\sigma=\mathbf{id}$ is 
\begin{align}
\prod_{\substack{1\leq i\leq n\\ b_j=\varphi(a_i)}}m_{ij}&=\prod_{\substack{1\leq i\leq n\\ b_j=\varphi(a_i)}}x_{a_i+b_j}=\prod_{1\leq i\leq n}x_{a_i+\varphi(a_i)}\nonumber\\
&=\prod_{x\in G}x^{m_\varphi(x)}.
\end{align}
Since $\varphi$ is an acyclic matching, this term is not canceled in the above expansion. Then the determinant of $M_{A,B}$ is nonzero and so $M_{A,B}$ is invertible.
\end{proof}
It is shown in \cite{9} that an abelian group satisfies the matching property if and only if it is torsion-free or cyclic of prime order. Later, it is proven in \cite{6} that this is the case for non-abelian groups as well. But similar classification for acyclic matching property has not been found yet. It is proven in \cite{9} that abelian torsion-free groups satisfy the acyclic matching property. But characterizing  the acyclic matching property in finite groups of prime order is still an unsolved problem. It is shown in \cite{2} that there are infinitely many prime $p$ for which $\mathbb{Z}/p\mathbb{Z}$ does not satisfy the acyclic matching property.  
In Section 5, we will employ a MATLAB code to numerically check  for which values of $p$, $\mathbb{Z}/p\mathbb{Z}$ have  the acyclic matching property.
 Its results show that for all $5<p<19$, $\mathbb{Z}/p\mathbb{Z}$ does not have the acyclic matching property. It seems that  for all $p>5$ one can find a pair $A,B\subset\mathbb{Z}/p\mathbb{Z}$ with $\#A=\#B$ and $0\not\in B$ and there is no acyclic matching from $A$ to $B$. Hence we have the following conjecture:
 \begin{conjecture}\label{c2.3}
 For any  prime $p>5$, $\mathbb{Z}/p\mathbb{Z}$ does not have the acyclic matching property.\\[-2mm]
  \end{conjecture}
Replacing the condition $0\not\in B$ with $A\cap (A+B)=\emptyset$ in the definition of the acyclic matching property guarantees that each bijection from $A$ to $B$ is a matching. It is also easy to verify that if every bijection $\varphi:A\to B$ is a matching, then $A\cap (A+B)=\emptyset$. Motivated by this observation, we  turn to a related, but weaker notion of acyclic matching property.
\begin{definition}\label{d2.4}
We say that $G$ possesses the weak acyclic  matching property if for every pair $A$ and $B$ of finite subsets of $G$ with $\#A=\#B$ and $A\cap(A+B)=\emptyset$ there exists at least one acyclic matching from $A$ to $B$. 
\end{definition}
Back to Losonczy's result (\cite[Theorem 4.1]{9}), every torsion-free abelian group has the acyclic matching property. Combining an observation from \cite{2} which states that there are infinitely many prime $p$ such that  $\mathbb{Z}/p\mathbb{Z}$ does not have the acyclic matching property with the fact that the weak acyclic matching property does not imply the acyclic matching property, one can probably find a prime $p$ such that $\mathbb{Z}/p\mathbb{Z}$ does not have the weak acyclic matching property.
\begin{conjecture}\label{c2.5}
There exists a  prime $p$ for which $\mathbb{Z}/p\mathbb{Z}$ does not satisfy the weak acyclic matching property.
\end{conjecture}
\begin{remark}\label{r2.6}
Note that the techniques used in the proof of the existence of infinitely many prime $p$ with $\mathbb{Z}/p\mathbb{Z}$ not having the acyclic matching property mentioned in \cite{2} might not be useful to apply for Conjecture \ref{c2.5} as the set $A=\{n^2:\, n\in (\mathbb{Z}/p\mathbb{Z})^*\}$ (or $A=\{2^i:\, 0\leq i\leq m-1\}\subset\mathbb{Z}/p\mathbb{Z}$) does not meet the condition $A\cap (A+B)=\emptyset$ in case we assume $A=B$. The reason for the first case is that the equation $x^2+y^2\equiv z^2$, modulo $p$,  always has  solution (indeed, $\frac{p-1}{2}$ non-equivalent solutions for $p\geq7$), see \cite{5}. The second case is immediate from the fact that the equation $2^x+2^y\equiv 2^z$, modulo $p$, has always solution. However, the fact that ``any acyclic matching from $A$ to itself has a fixed point provided that $\#A$ is odd'' should be helpful (\cite[Lemma 2.1]{2}).
\end{remark}
The condition $A\cap (A+B)=\emptyset$, where $A,B\subset \mathbb{Z}/n\mathbb{Z}$, seems strong enough to guarantee the existence of at least one acyclic matching from $A$ to $B$ for most integers $n$. Running a MATLAB code for the existence of the weak acyclic matching property in $\mathbb{Z}/n\mathbb{Z}$, we obtained that for all $1<n<23$, $\mathbb{Z}/n\mathbb{Z}$ has the weak acyclic matching property (see Section 5). Hence we have the following conjecture:
\begin{conjecture}
There are infinitely many $n$ for which $\mathbb{Z}/n\mathbb{Z}$ has the weak acyclic matching property.
\end{conjecture}
\section*{Linear Analogue of Matching Matrices}
We now give the linear analogue of matching matrices for subspaces in a field extension.  Throughout this section, we shall assume that $K\subset L$ is a field extension, $A,B\subset L$ are two $n$-dimensional $K$-subspaces of $L$ and $\mathcal{A}=\{a_1,\ldots,a_n\}$, $\mathcal{B}=\{b_1,\ldots,b_n\}$ are ordered bases of $A,B$, respectively. The Minkowski product $AB$ of $A$ and $B$ is defined as $AB:=\{ab:\, a\in A, b\in B\}$. Recall that Eliahou and Lecouvey have introduced the following notions for matchable bases of subspaces in a field extension \cite{7}. The ordered basis $\mathcal{A}$ is said to be {\it matched} to an ordered basis $\mathcal{B}$ of $B$ if 
\begin{align*}
 a^{-1}_iA\cap B\subset \langle b_1,\ldots,\hat{b}_i,\ldots,b_n\rangle,
\end{align*}
for each $1\leq i\leq n$, where $\langle b_1,\ldots,\hat{b}_i,\ldots,b_n\rangle$ is the vector space spanned  by  $\mathcal{B}\setminus\{b_i\}$. The subspace $A$ is {\it matched}  to the subspace $B$ if every basis of $A$ can be matched to a basis of $B$. Finally, the extension $K\subset L$ has the {\it linear matching property} if for every $n\geq1$ and any pair $A$, $B$ of $n$-dimensional $K$-subspaces of $L$ with $1\not\in B$, $A$ is matched to $B$. \\
For each $c\in L$ we associate a variable $x_c$. Consider now the commutative ring of polynomials with integer coefficients in all possible variables. Call it $D(L)$. \\
 Define the $n\times n$  matrix $M_{\mathcal{A},\mathcal{B}}=[m_{ij}]_{n\times n}$ over $D(L)$ as follows:
\begin{align}
m_{ij}=\begin{cases} x_{a_ib_j}&\mathrm{if}\;\; a_i^{-1}A\cap B\subset \langle b_1,\ldots,\hat{b}_j,\ldots,b_n\rangle;\\ 0&\mathrm{otherwise}.\end{cases}
\end{align}
\begin{definition}\label{d2.7}
The matrix $M_{\mathcal{A},\mathcal{B}}$ is called the linear  matching matrix from $\mathcal{A}$ to $\mathcal{B}$.
\end{definition}
Various results on matchings in groups  have recently been transposed to a linear setting [1, 2 and 10], and our objective now is to study the linear analogue of Proposition \ref{t2.2} for field extensions. In order to investigate linear matching matrices and their  relation with the acyclic matching property of vector subspaces we need the definition of the {\it linear acyclic matching property} from \cite{2}. The main key to define the linear acyclic matching property is strong matchings. Following Eliahou and  Lecouvey in \cite{7}, we say that a linear isomorphism $\varphi:A\to B$ is a {\it strong matching} from $A$ to $B$ if every ordered basis $\mathcal{A}$ of $A$ is matched to the basis $\mathcal{B}=\varphi(\mathcal{A})$ of $B$, under the bijection induced by $\varphi$. The following criterion for existence of a strong matching will be used to define the acyclicity property.
\begin{theorem}[{\cite[Theorem 6.3]{7}}]\label{t2.8}
Using the above notations, there is a strong matching from $A$ to $B$ if and only if $AB\cap A=\{0\}$.
\end{theorem}
Two linear isomorphisms $\varphi,\psi:A\to B$ are called {\it equivalent} if there exists a linear  automorphism $\phi:A\to A$ such that for all $a\in A$ one has $a\varphi(a)=\phi(a)\psi(\phi(a))$, and two strong matchings $\varphi,\psi:A\to B$ are equivalent if they are equivalent as linear isomorphisms. An {\it acyclic matching} from $A$ to $B$ is defined to be a strong matching $\varphi:A\to B$ such that for any strong matching $\psi:A\to B$ that is equivalent to $\varphi$, one has $\varphi=c\psi$, for some constant $c\in K$. Finally, it is said that the extension $K\subset L$ possesses the linear acyclic matching property if for every pair $A$ and $B$ of nonzero $n$-dimensional $K$-subspaces of $L$ with $AB\cap A=\{0\}$, there is at least one acyclic matching from $A$ to $B$. The main result concerning the linear acyclic  matching property is proven in \cite[Theorem 4.5]{2}, where it is shown that any purely transcendental field extension has the linear acyclic matching property.  The existence of acyclic matching is verified by the necessary condition if the following conjecture holds. This conjecture is basically the linear version of Proposition \ref{t2.2}.
\begin{conjecture}\label{c2.9}
If $\varphi:A\to B$ is an acyclic matching and $\mathcal{A}=\{a_1,\ldots,a_n\}$ is a basis of $A$,  then the matrix  $M_{\mathcal{A},\varphi(\mathcal{A})}$ is invertible.
\end{conjecture}
In what follows we provide  an  open problem  concerning the linear acyclic matching property in field extensions. It is shown in \cite{7} that if $K\subset L$ is a field extension, then $L$ has the linear matching property if and only if $L$ contains no proper finite-dimensional extension over $K$. But a similar classification for acyclic matching property is yet to be found. The above theorem implies that purely transcendental extensions and also finite extensions satisfy the linear matching property. As we mentioned above, the purely transcendental case is solved in \cite{2}. But characterizing field extensions of prime degree (or finite field extensions with no proper intermediate subfields in general) that satisfy the acyclic matching property remains an open problem. We have the following conjecture on the linear acyclic matching property of field extensions of prime degree.
\begin{conjecture}\label{c2.10}
There are infinitely many prime $p$  for which the following statement holds:\\
``There is a field extension $K\subset L$ with $[L:K]=p$, and $K\subset L$ does not admit  the linear acyclic  matching property."
\end{conjecture}
The following theorem gives  us  a tool for constructing matchable subspaces in field extensions that contain proper finite-dimensional intermediate subfields.
\begin{theorem}[{\cite[Theorem 5.5]{7}}]\label{t2.11}
Let $K\subset L$ be a field extension. Let $A,B\subset L$ be $n$-dimensional subspaces of $L$, with $1\not\in B$ and $n<n_0(K,L)$, where $n_0(K,L)$ stands for the smallest degree of an intermediate field extension $K\subsetneqq\mathbb{F}\subset L$. Then $A$ is matched to $B$.
\end{theorem}
In what follows, we provide another observation which can be used to investigate matchable subspaces in field extensions that contain proper finite-dimensional intermediate subfields.  We begin with a generalization for matchable bases in the following definition:
\begin{definition}\label{d2.13}
Let $\tilde{\mathcal{A}}=\{a_1,\ldots,a_m\}$ and $\tilde{\mathcal{B}}=\{b_1,\ldots,b_m\}$ be subsets of $A$ and $B$, respectively, (not linearly independent sets necessarily). Let $V$ be a $K$-subspace of $L$ and $\sigma\in \mathcal{S}_m$. We say that $\tilde{\mathcal{A}}$ is matched to $\tilde{\mathcal{B}}$ with respect to $(V,\sigma)$ if 
\begin{align}
a_i^{-1}A\cap B\subset V\oplus \langle b_1,\ldots,\hat{b}_{\sigma(i)},\ldots,b_m\rangle,
\end{align}
for all $1\leq i\leq m$.
\end{definition}
Note that if $m=n$, and  $\tilde{\mathcal{A}}$ and $\tilde{\mathcal{B}}$ are bases of $A$ and $B$, respectively, then $\tilde{\mathcal{A}}$ is matched to $\tilde{\mathcal{B}}$ in the usual sense provided $\tilde{\mathcal{A}}$ is matched to $\tilde{\mathcal{B}}$ with respect to $(\{0\},\mathbf{id})$, where $\{0\}$ denotes the trivial $K$-subspace of $L$ and $\mathbf{id}$ denotes the identity permutation on $\{1,\ldots,n\}$.
\begin{remark}
If $V$ and $W$ are $K$-subspaces of $L$ with $V\subset W$, then any matchable subsets with respect to $(V,\sigma)$ are matchable with respect to $(W,\sigma)$. 
\end{remark}
\begin{proposition}
Let $K\subset L$ and $K\subset E$ be two field extensions, $\sigma\in \mathcal{S}_n$ and $T\in \mathrm{Hom}_K(L,E)$. Then $\{Ta_1,\ldots,Ta_n\}$ is matched to $\{Tb_1,\ldots,Tb_n\}$ with respect to $(\{0\},\sigma)$ if and only if $\{a_1,\ldots,a_n\}$ is matched to $\{b_1,\ldots,b_n\}$ with respect to $(\ker T,\sigma)$.
\end{proposition}
\begin{proof}
The proof is trivial.
\end{proof}

\section{Local Matchings}
In this section, we first  provide a dimension  criterion to study linear local matchings in field extension similar to those in matchable bases. We also introduce the weakly locally  matchable subspaces and prove that matchable subspaces are weakly locally  matchable as well.  Recall  from \cite{3} that if $A,B\subset L$ are two $n$-dimensional $K$-subspaces, $n>1$, and $\tilde{A}$ and $\tilde{B}$ are two non-zero $m$-dimensional vector subspaces of $A$  and $B$, respectively, it is said that $\tilde{A}$ is {\it $A$-matched} to $\tilde{B}$, if  for any ordered  basis $\tilde{\mathcal{A}}=\{a_1,\ldots,a_m\}$ of $\tilde{A}$, there exists an ordered basis $\tilde{\mathcal{B}}=\{b_1,\ldots,b_m\}$ of $\tilde{B}$ for which $a_ib_i\not\in A$, for $i=1,\ldots,m$. In this case, it is said that $\tilde{\mathcal{A}}$ is $A$-matched to $\tilde{\mathcal{B}}$. It is also said that $A$ is {\it locally matched} to $B$ if for any intermediate subfield $K\subset H\subsetneqq L$ with $H\cap B\neq\{0\}$ and $aH\subset A$, for some $a\in A$, one can find a subspace $\tilde{A}$ of $A$ such that $\tilde{A}$ is $A$-matched to $H\cap B$. Finally, $K\subset L$ is said to have the {\it linear local matching property} if, for every $n\geq1$ and every pair of $n$-dimensional subspaces $A$ and $B$ of $L$ with $1\not\in B$, the subspace $A$ is locally matched to $B$. It is conjectured in \cite{3} that if $A$ is matched to $B$, in the field extension setting sense, then $A$ is locally matched to $B$. This problem is still unsolved. However, we will prove a dimension criterion for matchable subspaces  similar to the following dimension criterion to deal with local matchings. 
\begin{proposition}[{\cite[Proposition 3.1]{7}}]\label{p3.1}
Let $\mathcal{A}=\{a_1,\ldots,a_n\}$ be an ordered  basis of $A$. Then $\mathcal{A}$ can be matched to a basis of $B$ if and only if for all non-empty $J\subset\{1,\ldots,n\}$, we have 
\begin{align}
\dim\underset{i\in J}{\bigcap}(a_i^{-1}A\cap B)\leq n-\# J.
\end{align}
\end{proposition}
The following dimension criterion  will be used to study a weaker version of locally matchable subspaces called {\it weakly locally matchable subspaces}.
\begin{theorem}\label{t3.2}
Let $K\subset L$ be a field extension and $A$, $B$ be two non-zero $n$-dimensional $K$-subspaces of $L$ such that $A$ is matched to $B$. 
Let ${\mathcal{B}}=\{b_1,\ldots,b_n\}$ be an ordered  basis for $B$. Define  $A_b=\{a\in A:\, ab\in A\}$, for any $b\in B$. Then we have
\begin{align}\label{eq1}
\dim \underset{i\in J}{\bigcap} A_{b_i}\leq n-\#J,
\end{align}
for any $J\subset \{1,\ldots,n\}$.
\end{theorem}
\begin{proof}
If $\underset{i\in J}{\bigcap}A_{b_i}=\{0\}$, then \eqref{eq1} holds. Now, assume that $\underset{i\in J}{\bigcap} A_{b_i}$ is non-zero and $\{a_1,\ldots,a_t\}$ is an ordered  basis for it. As $a_k\in \underset{i\in J}{\bigcap}A_{b_i}$, for all $1\leq k\leq t$, then $a_kb_i\in A$, for all $i\in J$. This means $b_i\in a_k^{-1}A\cap B$, for all $i\in J$ and so $b_i\in \overset{t}{\underset{k=1}{\bigcap}}\left(a_k^{-1}A\cap B\right)$. This follows, $\# J\leq \dim  \overset{t}{\underset{k=1}{\bigcap}}\left(a_k^{-1}A\cap B\right)$. Since $A$ is matched to $B$, then by Proposition \ref{p3.1} we have $\dim  \overset{t}{\underset{k=1}{\bigcap}}\left(a_k^{-1}A\cap B\right)\leq n-t$. Thus, we totally obtain $t\leq n-\# J$ and this implies   that \eqref{eq1} is the case, as claimed.
\end{proof}
\section*{A Connection to the $m$-intersection Property}
The topic of upper bounds for the intersections of some families of sets has found some interest in literature, for example see \cite[Theorem 1.1]{6n}. Following Brualdi, Friedland and Pothen \cite{6n} we say that the family $\mathcal{J}=\{J_1,\ldots,J_t\}$ of subsets of $\{1,\ldots,n\}$, each of cardinality $m-1$ satisfies the {\it $m$-intersection property} provided 
\[\#\underset{i\in J}{\bigcap}J_i\leq m-\# J,\]
for any $J\subset \{1,\ldots,t\}$, $J\neq\emptyset$. We generalize this notion as follows:
\begin{definition}
The family of $\mathcal J=\{J_1,\ldots,J_t\}$ of subsets of $\{1,\ldots,n\}$, each of cardinality $\leq m-1$ satisfies the {\it weak $m$-intersection property} provided 
\[\# \underset{i\in J}{\bigcap} J_i\leq m-\#J,\]
for all $J\subset\{1,\ldots,t\}$, $J\neq\emptyset$.
\end{definition}
We now formulate the linear analogue of this concept.
\begin{definition}
Let $E$ be an $n$-dimensional $K$-vector space and $\mathcal{E}=\{E_1,\ldots,E_t\}$ be a family of  subspaces of $E$ with $\dim_KE_i\leq m-1$. We say that $\mathcal{E}$ satisfies the weak   linear $m$-intersection property provided
\[\dim\underset{i\in J}{\bigcap}E_i\leq m-\#J,\]
for any $J\subset \{1,\ldots,t\}$, $J\neq\emptyset$.
\end{definition}
Back to matchable subspaces setting, the possibility of matching given bases of $A$ and $B$ is reformulated in terms of the weak linear $n$-intersection property as   follows:
\begin{theorem}\label{t3.5}
Let $K,\,L,\,A,\,B$ and $A_b$ as Theorem \ref{t3.2}. Let $\mathcal{A}=\{a_1,\ldots,a_n\}$ and $\mathcal{B}=\{b_1,\ldots,b_n\}$ be bases for $A$ and $B$, respectively.  Then the family $\left\{A_{b_1},\ldots,A_{b_n}\right\}$ of subspaces of $A$ has the weak linear $n$-intersection property provided $\mathcal A$ is matched to $\mathcal B$.
\end{theorem}
\begin{remark}
Note  that since $\mathcal A$ is matched to $\mathcal B$, then $a_i\not\in A_{b_i}$, for all $1\leq i\leq n$. This implies  $\dim A_{b_i}\leq n-1$, for all $1\leq i\leq n$.
\end{remark}
The interested reader is encouraged to see \cite[Section 5.3]{9n} for more results on the $m$-intersection family which is used to study the sparse basis problem. See also \cite{1}.

\section*{The Weak $m$-intersection Property  and Matchable Subsets}
In contrast to precedent routine that results first are proven in the group setting and then are generalized to field extensions, we investigate the analogue of Theorem \ref{t3.5} in groups.  We begin with the following lemma which is the analogue of Proposition \ref{p3.1} in groups. 
\begin{lemma}\label{l3.7}
Let $G$ be an abelian group and $A=\{a_1,\ldots,a_n\}$, $B=\{b_1,\ldots,b_n\}$ be two subsets of $G$. Then $A$ is matched to $B$ if and only if
\[\#\underset{i\in J}{\bigcap}\left((A-a_i)\cap B\right)\leq n-\#J,\]
for all non-empty $J\subset \{1,\ldots,n\}$. 
\end{lemma}
In our proof, we will use the marriage theorem of Hall \cite{14n} which states that if $E$ is a set and $\mathcal{E}=\{E_i\}_{i=1}^n$ is a family of finite subsets of $E$, then $\mathcal{E}$ admits a system of distinct representatives if and only if 
\[\#\underset{i\in J}{\bigcup} E_i\geq\#J,\]
for all non-empty subsets $J\subset\{1,\ldots,n\}$.
\begin{proof}[Proof of Lemma \ref{l3.7}]
$\Rightarrow$ Assume first that $A$ is matched to $B$. Then there exists a bijection $\varphi:A\to B$ such that $a+\varphi(a)\not\in A$. This implies 
\[(A-a_i)\cap B\subset B\setminus \{\varphi(a_i)\},\]
for all $1\leq i\leq n$. Therefore, for any non-empty $J\subset \{1,\ldots,n\}$ we have 
\[\underset{i\in J}{\bigcap}\left((A-a_i)\cap B\right)\subset \underset{i\in J}{\bigcap}\left(B\setminus\{\varphi(a_i)\}\right)=B\setminus\{\varphi(a_i):\, i\in J\}.\]
It follows that $\#\underset{i\in J}{\bigcap}\left((A-a_i)\cap B\right)\leq n-\#J$, as claimed.\\
$\Leftarrow$ Assume that for all non-empty $J\subset\{1,\ldots,n\}$, we have 
\[\#\underset{i\in J}{\bigcap}\left((A-a_i)\cap B\right)\leq n-\#J.\]
Taking the complement in $B$, we get 
\[\#\underset{i\in J}{\bigcup}\left((A-a_i)\cap B\right)^c\geq \#J.\]
Using the marriage theorem of Hall, the above bounds imply the existence of a permutation $\sigma\in\mathcal{S}_n$ such that $b_{\sigma(i)}\in\left((A-a_i)\cap B\right)^c$. In other words $a_i+b_{\sigma(i)}\not\in A$. This induces  the bijection $\varphi:A\to B$ via the rule $\varphi(a_i)=b_{\sigma(i)}$ is a matching, as desired.
\end{proof}
The following observation is  analogous to Theorem \ref{t3.2} in the group setting. 
\begin{proposition}\label{p3.8}
Let $G$, $A$ and $B$ be as Lemma \ref{l3.7}. For any $1\leq i\leq n$, define $A_{b_i}=\{a\in A:\, a+b_i\in A\}$. Let $A$ be matched to $B$. Then 
\begin{align}\label{eq8}
\#\underset{i\in J}{\bigcap} A_{b_i}\leq n-\#J,
\end{align}
for any non-empty $J\subset \{1,\ldots,n\}$.
\end{proposition}
\begin{proof}
If $\underset{i\in J}{\bigcap}A_{b_i}=\emptyset$, then \eqref{eq8} holds. Assume that $\underset{i\in J}{\bigcap}A_{b_i}$ is non-empty and without loss of generality assume that $\underset{i\in J}{\bigcap}A_{b_i}=\{a_1,\ldots,a_t\}$. As $a_k\in \underset{i\in J}{\bigcap}A_{b_i}$, for all $1\leq k\leq t$, then $a_k+b_i\in A$, for all $i\in J$ and so $b_i\in \overset{t}{\underset{k=1}{\bigcap}}\left((A-a_k)\cap B\right)$. It follows that  $\#J\leq \#\overset{t}{\underset{k=1}{\bigcap}}\left((A-a_k)\cap B\right)$. Since $A$ is matched to $B$, then by Lemma \ref{l3.7}  we have $\#\overset{t}{\underset{k=1}{\bigcap}}\left((A-a_k)\cap B\right)\leq n-t$. Thus we totally obtain $t\leq n-\#J$ and so \eqref{eq8} is the case, as claimed. 
\end{proof}
As an easy consequence, we get the following result.
\begin{corollary}
Let $G,\, A,\,B$ and  $A_{b_i}$ be as Proposition \ref{p3.8}. Then the family $\{A_{b_1},\ldots,A_{b_n}\}$ of subsets of $A$ has the weak  linear $n$-intersection property provided that $A$ is matched to $B$.
\end{corollary}
\section*{A Few Notations}
We shall use the following standard notation. We denote 
\begin{align*}
A^*=\{\varphi:A\to K:\; \varphi \text{ is linear }\}
\end{align*}
the {\it dual} of $A$. Moreover for any subspace $E$ of $A$, we denote 
\begin{align*}
E^\perp=\{\varphi\in A^*:\; E\subset \ker\varphi\}
\end{align*}
the {\it orthogonal} of $E$ in $A^*$. We will also use the fact that $E\oplus E^\perp=A$. See \cite{13nn} for more details.\\[5mm]
Let $\mathcal{E}=\{E_i\}_{i=1}^n$ be  a family of subspaces of a finite-dimensional $K$-vector space $E$. A {\it free transversal} for  $\mathcal{E}$ is a linearly independent set of vectors $\{x_1,\ldots,x_n\}$ in $E$ satisfying $x_i\in E$, for all $i=1,\ldots,n$. It is shown by Rado in \cite{11} that  $\mathcal E$ admits a free transversal if and only if
\begin{align}
\dim\underset{i\in J}{\sum} E_i\geq\# J,
\end{align}
for any non-empty  $J\subset \{1,\ldots,n\}$. This result is actually very similar the marriage theorem of Hall.  
\section*{Back to the Linear Setting}
Using Rado's theorem we have the following observation:
\begin{proposition}
Let $E$ be an $n$-dimensional vector space over $K$ and $\mathcal{E}=\{E_1,\ldots,E_k\}$ be a family of vector subspaces of $E$ such that for any non-empty $J\subset \{1,\ldots,k\}$ we have 
\begin{align}
\dim \underset{i\in J}{\bigcap}E_i\leq n-\#J.
\end{align}
Then there exist vector subspaces $\tilde{E}_i$ of $E$ such that $E_i\subset \tilde{E}_i$ and for any non-empty $J\subset \{1,\ldots,k\}$, we have
\begin{align}
\dim \underset{i\in J}{\bigcap}\tilde{E}_i=n-\# J.
\end{align}
\end{proposition}
\begin{proof}
Since $\dim \underset{i\in J}{\bigcap}E_i\leq n-\# J$, we have $\dim \underset{i\in J}{\sum}E_i^\perp\geq\# J$. Using Rado's theorem the family $\tilde{\mathcal{E}}=\{E_1^\perp,\ldots,E_k^\perp\}$ admits a free transversal. Let $(a_1,\ldots,a_k)$ be a free transversal for $\tilde{\mathcal{E}}$. Then $\dim \underset{i\in J}{\sum}\langle a_i\rangle=\# J$. Setting $\tilde{E}_i=\langle a_i\rangle^\perp$  we get  $\dim \underset{i\in J}{\bigcap}\tilde{E}_i=n-\# J$.
\end{proof}
Combining Theorem \ref{t3.2} with the above observation yields the following corollary:
\begin{corollary}
Let $K,L,A,B,\mathcal{B}$ and $A_{b_i}$ be as Theorem \ref{t3.2}. Then, there exist subspaces $\tilde{A}_i$ of $A$ such that $A_{b_i}\subset \tilde{A}_i$ and $\dim \underset{i\in J}{\bigcap}\tilde{A}_i=n-\# J$, $J\subset\{1,\ldots,n\}$ and $J\neq\emptyset$.
\end{corollary}
With the notion  of locally matched subspaces at hand, we now introduce weakly locally matched subspaces.
\begin{definition}
Let $K\subset L$ be a field extension and $A, B$ be two $n$-dimensional $K$-subspaces of $L$ with $n>1$. Let $\tilde{\mathcal{A}}$ and $\tilde{\mathcal{B}}$ be two $m$-dimensional $K$-subspaces of $A$ and $B$, respectively. We say that $\tilde{A}$ is weakly $A$-matched to $\tilde{B}$ if there exist ordered bases $\{a_1,\ldots,a_m\}$ and  $\{b_1,\ldots,b_m\}$ of $\tilde{A}$ and $\tilde{B}$,  respectively, such that   $a_ib_i\not\in A$, for $i=1,\ldots,m$. We also say that $A$ is weakly locally matched to $B$ if for any  intermediate subfield $K\subset H\subsetneqq L$ with $H\cap B\neq\{0\}$ and $aH\subset A$, for some $a\in A$, one can find a subspace $\tilde{A}$ of $A$ such that $\tilde{A}$ is weakly  $A$-matched to $H\cap B$. 
\end{definition}
Combining Theorem \ref{t3.2} with a construction due to Eliahou-Lecouvey \cite{7} we arrive at the following result:
\begin{theorem}\label{t3.6}
Let $K,\, L,\, A$ and $B$ be as in  Theorem \ref{t3.2}. Then $A$ is weakly locally matched to $B$ provided $A$ is matched to $B$.
\end{theorem}
\begin{proof}
  Assume that $H$ is an intermediate subfield $K\subset H\subsetneqq L$ satisfying $H\cap B\neq0$ and $aH\subset A$, for some $a\in A$.   We show that there exists a subspace $\tilde{A}$ of $A$ such that $\tilde{A}$ is weakly $A$-matched to $H\cap B$.  Let $\tilde{\mathcal{B}}=\{b_1,\ldots,b_m\}$  be an ordered  basis for $H\cap B$ and extend it to the ordered  basis $\mathcal{B}=\{b_1,\ldots,b_n\}$ for $B$. Using Theorem  \ref{t3.2}  we have 
\begin{align*}
\dim \underset{i\in j}{\bigcap} A_{b_i}\leq n-\#J
\end{align*}
for all non-empty  $J\subset \{1,\ldots,n\}$. Taking the orthogonal in the dual space $A^*$, we get 
\begin{align*}
\dim\left(\underset{i\in J}{\bigcap} A_{b_i}\right)^\perp\geq \# J,
\end{align*}
and hence 
\begin{align*}
\dim \underset{i\in j}{\sum}A_{b_i}^\perp \geq \#J.
\end{align*}
By the Rado's theorem, the above dimensions bound implies the existence of a free transversal 
\[\varphi_1,\ldots,\varphi_n\in A^*\]
for the system of subspaces $\left\{A_{b_i}^\perp\right\}_{i=1}^n$. In other words, we have $\varphi_i\in A_{b_i}^\perp$ for $1\leq i\leq n$, and $\{\varphi_1,\ldots,\varphi_n\}$ is  an ordered  basis of $A^*$. Let $\mathcal{A}=\{a_1,\ldots,a_n\}$ be the unique ordered basis of $A$ whose dual basis $\mathcal{A}^*$ equals $\{\varphi_1,\ldots,\varphi_n\}$,  i.e. such that $a_i^*=\varphi_i$ for all $i$. Then we have
\begin{align}\label{eqn8}
a_i^*(A_{b_i})=\{0\}.
\end{align}
Set $\tilde{A}:=\langle a_1,\ldots,a_m\rangle$, the $K$-subspace of $L$ spanned by $a_1,\ldots,a_m$. By \eqref{eqn8}, we have $a_ib_i\not\in A$, for all $i=1,\ldots,m$. Then $\tilde{A}$ is  weakly  $A$-matched to $H\cap B$. This follows that $A$ is weakly locally matched to $B$.
\end{proof}
\begin{remark}
We hope that the techniques presented in the proof of Theorem  \ref{t3.2} and Theorem \ref{t3.6} can be used to solve the main problem of local matchings which states that matchable subspaces are locally matchable \cite[Remark 5.6]{3}
\end{remark}
 
\section{A Dimension Criterion For Primitive Matchable Subspaces}
It is shown in \cite{9} that for a non-trivial finite cyclic group $G$ and finite non-empty subsets $A$, $B$ of it  with $\# A=\# B$ there exists a matching from $A$ to $B$ if every element of $B$ is a generator of $G$. The linear analogue of this result is given in \cite{3} as the following theorem. 
\begin{theorem}\label{t4.1}
Let $K\subset L$ be a separable field extension and $A$ and $B$ be two $n$-dimensional $K$-subspaces in $L$ with $n>1$ and $B$ is a primitive $K$-subspace of $L$. Then $A$ is matched to $B$.
\end{theorem}
Note that we say that $K\subset L$ is a {\it simple} extension if $L=K(\alpha)$, for some $\alpha\in L$. Also, if $B$ is a $K$-subspace of $L$ such that $K(b)=L$, for any $b\in B\setminus\{0\}$, we say that $B$ is a {\it primitive} $K$-subspace of $L$. 
Motivated by the above theorem, a natural question to ask is that how large can the primitive subspace $B$ be? We answer this question in Theorem \ref{t4.2}. We actually prove that $\dim B\leq[L:K]-n(K,L)$, where $n(K,L)$  denotes the largest degree of an intermediate  field extension $K\subset F\subsetneqq L$ over $K$. 
\begin{theorem}\label{t4.2}
Let $K$ be an infinite field and $K\subset L$ be a finite simple field extension. Then we have
\begin{align}\label{eq10}
m(K,L)=n-n(K,L),
\end{align}
where $n=[L:K]$, $n(K,L)=\max\big\{[F:K]:\,F$ is an intermediate  subfield of $K\subset L$ with $F\neq L\big\}$ and $m(K,L)=\max\big\{\dim  V:\, V\subset L$ is a 
$K$-subspace and $K(a)=L,$ for any $a\in V\setminus \{0\}\big\}$.
\end{theorem}
Note that in the proof of the above theorem, we  use the fact that ``finite union of lower-dimensional subfields of number fields (considered as vector spaces over $\mathbb{Q}$) do not cover the field'', namely,  ``a vector space over an infinite base field cannot be written as a finite union of its proper subspaces'' as a key  ingredient.     The  Artin's theorem on primitive elements  which states that ``if $K\subset L$ is a finite field extension and $K$ is infinite, then $K\subset L$ is a simple extension if and only if there are only a finite number of intermediate subfields of $K\subset L$'' is also an engine behind our proof  \cite[Theorem 22.1.22]{10}.
\begin{proof}
Artin's theorem on primitive elements implies that $K\subset L$ has only finitely many intermediate subfields. Assume that $\{\mathbb{F}_i\}_{i=1}^r$ is the family of all proper intermediate subfields of $K\subset L$. Define $\psi=\big\{V\subset L: \, V$ is a $K$-vector subspace of $L$ and $K(x)=L$, for any $x\in V\setminus\{0\}\big\}$. Without loss of generality assume that $n(K,L)=[\mathbb{F}_1:K]$. Choose $V\in\psi$ for which $\dim V=m(K,L)$. It follows from the above mentioned linear algebra result that $L\neq\overset{r}{\underset{i=1}{\bigcup}}\mathbb{F}_i$. Choose $a_1\in L\setminus \overset{r}{\underset{i=1}{\bigcup}}\mathbb{F}_i$ and define $\mathbb{F}_i^{(1)}=\mathbb{F}_i\oplus\langle a_1\rangle$, for $1\leq i\leq r$. So $\{\mathbb{F}_i^{(1)}\}_{i=1}^r$ is a finite family of proper $K$-subspaces of $L$. Next, choose $a_2\in L\setminus \overset{r}{\underset{i=1}{\bigcup}}\mathbb{F}_i^{(1)}$ and define $\mathbb{F}_i^{(2)}=\mathbb{F}_i^{(1)}\oplus\langle a_2\rangle$, for $1\leq i\leq r$. Likewise, we get $L\neq \overset{r}{\underset{i=1}{\bigcup}}\mathbb{F}_i^{(2)}$. Continuing in this manner, we obtain a finite family of $K$-subspaces $\mathbb{F}_i^{(j)}$ of $L$, $1\leq i\leq  r$ and $1\leq j\leq n-[\mathbb{F}_1:K]$. Consider the $K$-subspace $W$ of $L$ spanned by $\left\{a_1,a_2,\ldots,a_{n-[\mathbb{F}_1:K]}\right\}$.  For any $x\in W\setminus\{0\}$, we have $x\not\in \overset{r}{\underset{i=1}{\bigcup}}\mathbb{F}_i$. This follows that $K(x)=L$. Therefore $W\in\psi$. This implies $m(K,L)\geq \dim W= n-[\mathbb{F}_1:K]=n-n(K,L)$. We claim that $m(K,L)=n-n(K,L)$ because otherwise $\dim V>\dim W$ which yields $[\mathbb F_1:K]+\dim V>n$. This follows $\mathbb F_1\cap V\neq\{0\}$ which is a contradiction (note that if $x\in(\mathbb F_1\cap V)$ is nonzero then $x$ cannot be a primitive element of $K\subset L$.) Therefore, $\dim V=\dim W$ and this yields \eqref{eq10}.
\end{proof}

\begin{remark}\label{r4.3}
Note that Theorem \ref{t4.2} is  probably valid even if $K$ is a finite field. To work on the finite base field case, a theorem by Lenstra-Schoof which states that ``for any prime power $q$ and positive integer $m$ there exists a primitive normal basis of $\mathbb{F}_{q^m}$ over $\mathbb{F}_q$" shall be helpful \cite{8}. Further investigations along this line could prove to be worthwhile. 
\end{remark}

\begin{remark}
Note that Theorem \ref{t4.2} has a very similar version in the group setting. Given a cyclic group $G$ of order $n=p^k$, where $p$ is prime and $k\in\mathbb{N} $. Define $n(G)=\max\big\{\#H:\, H$ is a proper subgroup of $G\big\}$ and $m(G)=\#\{a\in G:\, \langle a\rangle=G\}$ . Clearly, $n(G)=p^{k-1}$ and $m(G)=\phi(n)$, where $\phi$ stands for the Euler totient function. Then we have
\begin{align*}
m(G)=\phi(n)=n\left(1-\frac{1}{p}\right)=n-p^{k-1}=n-n(G).
\end{align*}
\end{remark}

\section{Computer Program}
In this section, we employ three algorithms to investigate acyclic matchings in finite cyclic groups. In the first two algorithms, we input $n$ and the expected output is that whether $\mathbb{Z}/n\mathbb{Z}$ has the weak acyclic matching property or not. In the third algorithm, we input $p$, where $p$ is a prime, and the expected output is that whether $\mathbb{Z}/p\mathbb{Z}$ has the acyclic matching property or not. \\

\noindent{\large\textbf{Weak Acyclic Matching Property Algorithm Pseudo Code}}
$$\mathbb{S}=\{ (A,B)\quad : \quad A\cap (A+B)=\emptyset \quad\&\quad \#A=\#B\}$$

\begin{algorithm}
\caption{Checking for Weak Acyclic Matching Property}
\begin{algorithmic}[1]
\Function{WeakAcyclicMatchingPropertyCheck}{$n$}
\For{$(A,B)\in \mathbb{S}$}
\If{NoAcyclicMatchingCheck($(A,B)$)}
\State \Return $(A,B)$
\EndIf
\EndFor
\State \Return $(\emptyset,\emptyset)$

\EndFunction
\end{algorithmic}
\end{algorithm}

Let $\mathbb{B}(A,B)$ be the set of all bijections between $A$ and $B$.

Let $\mathbb{M}$ be a multiset of multiplicity functions.

\begin{algorithm}
\caption{Checking whether there exists no Acyclic Matching for $(A,B)$}

\begin{algorithmic}[1]
\Function{NoAcyclicMatchingCheck}{$(A,B)$}
\For{$b\in \mathbb{B}(A,B)$}
\State $\mathbb{M}=\mathbb{M}\cup m_b$
\EndFor
\If{$\mathbb{M}$ has an element of multiplicity 1}
\State \Return False
\Else
\State \Return True
\EndIf

\EndFunction
\end{algorithmic}
\end{algorithm}
 
 \noindent{\large\textbf{Acyclic Matching Property Algorithm Pseudo Code}}
 \[\mathbb{T}=\{(A,B):\; \#A=\#B\;\&\; 0\not\in B\}\]
 \begin{algorithm}
\caption{Checking for Acyclic Matching Property}
\begin{algorithmic}[1]
\Function{AcyclicMatchingPropertyCheck}{$p$}
\For{$(A,B)\in \mathbb{T}$}
\If{NoAcyclicMatchingCheck($(A,B)$)}
\State \Return $(A,B)$
\EndIf
\EndFor
\State \Return $(\emptyset,\emptyset)$
\EndFunction
\end{algorithmic}
\end{algorithm}
 
\section*{Simulation Results}
Algorithms 1 and 2 show that for all $1<n<23$, $\mathbb{Z}/n\mathbb{Z}$ has the weak acyclic matching property. 
We show the results of algorithm 3 in the following table. In the second column, ``Yes" signifies the existence of the acyclic matching property and otherwise, we use ``No". In case $\mathbb{Z}/p\mathbb{Z}$ does not have the acyclic matching property, in the third column we provide two subsets $A, B$ of $\mathbb{Z}/p\mathbb{Z}$ of the same cardinality and $0\not\in B$ for which there is no acyclic matching from $A$ to $B$. \\
Since the number of subsets to check for the existence of acyclic matchings increases exponentially with $p$, we were not able to run the code beyond $p=19$.

\begin{table}[h!]
\centering
\begin{tabular}{|c|c|c|}
\hline
$p$& Acyclic matching property & $A$ and $B$\\
\hline
2& Yes & -\\
\hline
3& Yes & -\\
\hline
5& Yes & -\\
\hline
\multirow{2}{*}{7}& \multirow{2}{*}{No} & $A=\{0,4,6\}$\\
& & $B=\{3,5,6\}$\\
\hline
\multirow{2}{*}{11}& \multirow{2}{*}{No} & $A=\{0,6,8,9,10\}$\\
& & $B=\{5,7,8,9,10\}$\\
\hline
\multirow{2}{*}{13}& \multirow{2}{*}{No} & $A=\{0,6,8,9,10,11,12\}$\\
& & $B=\{3,5,7,9,10,11,12\}$\\
\hline
\multirow{2}{*}{17}& \multirow{2}{*}{No} & $A=\{0,8,10,11,12,13,14,15,16\}$\\
& & $B=\{3,5,7,9,11,13,14,15,16\}$\\
\hline
\multirow{2}{*}{19}& \multirow{2}{*}{No} & $A=\{0,8,11,12,13,14,15,16,17,18\}$\\
& & $B=\{5,7,11,12,13,14,15,16,17,18\}$\\
\hline
\end{tabular}
\end{table}

\section{A Possible Research Problem}
There exists a generalization of the group concept called {\it $n$-group}. The first formal definition of $n$-groups was given by W. D\"ornte in \cite{2n}. Indeed, an $n$-group is a generalization of a group to a set $G$ with an $n$-ary operation instead of a binary operation. A short review of basic results on $n$-groups can be found in \cite{3n}. Also see \cite{1n}. Following D\"ornte we say that a non-empty set $G$ together with an $n$-ary operation $f:G^n\to G$ is an {\it $n$-groupoid}. We say that this operation is {\it associative} if for all $a_1,a_2,\ldots,a_{2n-1}\in G$ we have 
\begin{align*}
f\left(f(a_1,\ldots,a_n),a_{n+1},\ldots,a_{2n-1}\right)&=f\left(a_1,f(a_2,\ldots,a_{n+1}),a_{n+2},\ldots,a_{2n-1}\right)\\
&\vdots\\
&=f\left(a_1,\ldots,a_{n-1},f(a_n,\ldots,a_{2n-1})\right),
\end{align*}
and in this case $G$ is called an {\it $n$-semigroup}. If for all $a_1,\ldots,a_n$ in $G$ the equations:
\begin{align*}
&f(x_1,a_1,\ldots,a_{n-1})=a_n,\\
&f(a_1,x_2,a_2,\ldots,a_{n-1})=a_n,\\
&\vdots\\
&f(a_1,\ldots,a_{n-1},x_n)=a_n,
\end{align*}
have unique solutions, then $G$ is called an {\it $n$-quasigroup}. We call $G$ an {\it $n$-group} if $G$ is an $n$-semigroup and $n$-quasigroup. We now generalize the concept of matchings in groups to matchings in $n$-groups as follows:
\begin{definition}
Let $G$ be a $2n$-group and $A$ and $B$ be two finite subsets of $G$ with the same cardinality. A bijection $\varphi:A\to B$ is called a matching if for any $(a_1,\ldots,a_n)\in A^n$ we have $f(a_1,\ldots,a_n,\varphi(a_1),\ldots,\varphi(a_n))\not\in A$. In this case we say $A$ is matched to $B$.
\end{definition}
\begin{remark}
If $G$ is a 2-group, a matching in  $G$ is nothing but a matching in the usual sense. We hope that the results presented on matchings in groups have more general applicability, especially in the direction of generalizing these statements  to $2n$-groups. 
\end{remark}
\section*{Acknowledgment}
We are deeply grateful to Prof. Shmuel Friedland for several helpful discussions, useful comments  and reading an earlier version of the present paper. We would also thank the anonymous referee for making several useful comments.

\end{document}